\documentclass[11pt]{amsart}

\usepackage{amssymb, amsthm, amsmath, gensymb, dsfont}
\usepackage{graphicx, comment}
\usepackage[small]{caption}
\usepackage{subcaption}
\usepackage{epsfig}
\usepackage{tikz, pgfplots, float}

\newcommand{\later}[1]{}
\newcommand{\old}[1]{}

\usepackage{amsfonts}

\usepackage[utf8]{inputenc}
\usepackage{fullpage}
\usepackage{framed}

\usepackage{enumerate}
\usepackage{url}
\usepackage[breaklinks]{hyperref}
\usepackage{cleveref}
\hypersetup{
	colorlinks = true, % colors links instead of ugly boxes
	urlcolor = cyan, % color for external hyperlinks
	linkcolor = teal, % color of internal links
	citecolor = cyan % color of citations
}

\newtheorem{theorem}{Theorem}
\newtheorem{definition}[theorem]{Definition}
\newtheorem{lemma}[theorem]{Lemma}
\newtheorem{claim}[theorem]{Claim}
\crefname{claim}{claim}{claims}

\newtheorem{proposition}[theorem]{Proposition}
\newtheorem{corollary}[theorem]{Corollary}

\newtheorem{conjecture}[theorem]{Conjecture}

\newcommand{\eps}{\varepsilon}

\newcommand{\NN}{\mathbb{N}} %  set of natural numbers
\newcommand{\ZZ}{\mathbb{Z}} %  set of integer numbers
\newcommand{\RR}{\mathbb{R}} %  set of real numbers
\usepackage{tikz}
\tikzstyle{vtx} = [circle, fill, inner sep=0.7]

\title{The number of regular simplices in higher dimensions}

\author{Felix Christian Clemen}
\address{Felix Christian Clemen \newline Department of Mathematics and Statistics, University of Victoria, Victoria, B.C., Canada}
\email{fclemen@uvic.ca}

\author{Adrian Dumitrescu}
\address{Adrian Dumitrescu \newline Algoresearch L.L.C., Milwaukee, WI, USA \newline Alfr\'ed R\'enyi Institute of Mathematics, Budapest, Hungary \newline Research Institute of the University of Bucharest, Bucharest, Romania}
\email{ad.dumitrescu@algoresearch.org}

\author{Dingyuan Liu}
\address{Dingyuan Liu \newline Institute of Algebra and Geometry, Karlsruhe Institute of Technology, Karlsruhe, Germany}
\email{liu@mathe.berlin}

\date{}

\begin{document}

\maketitle

\begin{abstract}
We study the extremal function $S^k_d(n)$, defined as the maximum number of regular $(k-1)$-simplices spanned by $n$ points in $\mathbb{R}^d$. For any fixed $d\geq2k\geq6$, we determine the asymptotic behavior of $S^k_d(n)$ up to a lower-order term. In particular, when $k=3$, we determine the exact value of $S^3_d(n)$, for all even dimensions $d\geq6$ and sufficiently large $n$. This resolves a conjecture of Erd\H{o}s in a stronger form. The proof leverages techniques from hypergraph Tur\'an theory and linear algebra.
\end{abstract} 


\section{Introduction}
\label{sec:intro}
In 1946, Erd\H{o}s~\cite{Er46} posed a deceptively simple question: What is the maximum number of times the unit distance can occur among $n$ points in the plane? He proved a lower bound of $n^{1+\Omega(1/\log\log{n})}$ and conjectured that his lower bound is essentially optimal. See~\cite{Sz16} for a historical account of the conjecture. Contrary to long-standing belief, this conjecture was recently disproved by an internal model of OpenAI~\cite{AI26}. By generalizing Erd\H{o}s's original construction, the AI model established a lower bound of the form $n^{1+c}$, where $c>0$ is an absolute constant. See also~\cite{NA26} for a simplified proof and further remarks. Shortly thereafter, Sawin~\cite{WS26} improved the constant to $c=0.014$. The pursuit of this problem has inspired substantial research and also led to the formulation of numerous related questions in combinatorial geometry, all revolving around a common theme: determining the maximum number of certain geometric configurations that can be found among $n$ points in Euclidean space. In this paper, we address one such problem.

The following question was raised repeatedly by Erd\H{o}s and Purdy~\cite{Er75,EP71,EP75,EP76}: Given $d,n\in\NN$, what is the maximum number $T_d(n)$ of equilateral triangles that can be determined by $n$ points in $\RR^d$? There is a substantial body of work addressing this problem and its generalizations, see, e.g.,~\cite{AF20,AAPS07,AS02,ATT97,AS05,EP75,Pu88,SSZ13}. \'Abrego and Fern\'andez-Merchant~\cite{AF20} proved that $T_2(n)=\Theta(n^2)$, which also yields the lower bound $T_4(n)\geq T_3(n)=\Omega(n^2)$. For $d=5$, the same authors~\cite{AF20} established a better lower bound $T_5(n)=\Omega(n^{7/3})$. On the other hand, Sharir, Sheffer, and Zahl~\cite{SSZ13} proved that $T_3(n)=O(n^{15/7})$, and Agarwal, Apfelbaum, Purdy, and Sharir~\cite{AAPS07} showed that $T_4(n)=O(n^{12/5})$ and $T_5(n)=O(n^{8/3})$.

The first non-degenerate case, that is, when $\liminf_{n\to\infty}T_{d}(n)/\binom{n}{3}>0$, appears for $d=6$. Erd\H{o}s and Purdy~\cite{EP75} showed that $T_6(n)\geq n^3/27-O(n^2)$ by placing $n$ points evenly on three pairwise orthogonal circles. Regarding the upper bound, Erd\H{o}s and Purdy~\cite{EP75} said ``It would be interesting even to show $T_6(n)\leq(1/6-\eps)n^3$.'' In a later paper~\cite{Er94}, Erd\H{o}s proposed the following more precise conjecture.
\begin{conjecture}[Erd\H{o}s~\cite{Er94}]
\label{conj:triangle}
Any $n$ points in $\RR^6$ can span at most $n^3/27+o(n^3)$ equilateral triangles, i.e., $T_6(n)\leq n^3/27+o(n^3)$.
\end{conjecture}

We prove this conjecture in a more general form. For an integer $k\geq3$, a regular $(k-1)$-simplex is a set of $k$ vertices that are equidistant from one another. Let $S^k_d(n)$ denote the maximum number of regular $(k-1)$-simplices spanned by $n$ points in $\RR^d$. Our first theorem determines the asymptotic behavior of $S^k_d(n)$ for all fixed $d\geq2k\geq6$.

\begin{theorem}
\label{simplex}
Let $d\geq2k\geq6$ be fixed integers and let $r=\lfloor{d/2}\rfloor$. Then
\[S^k_{d}(n)=\binom{r}{k}\left(\frac{n}{r}\right)^k+o(n^k).\]
\end{theorem}

Since $T_d(n)=S^3_d(n)$,~\Cref{simplex} implies~\Cref{conj:triangle} upon taking $d=6$ and $k=3$. Our method also yields a stronger version of~\Cref{conj:triangle}. That is, we are able to determine the exact value of $T_d(n)$ for every even integer $d\geq6$ and sufficiently large $n$.

\begin{theorem}
\label{triangle_even}
Let $r\geq3$ be a fixed integer. Let $n\in\NN$ be sufficiently large and $p$ be the remainder of $n$ divided by $2r$. Then
\[T_{2r}(n)=\sum_{1\leq i<j<k\leq r}n_in_jn_k+\sum_{i\in[r]}\left((n_i-\mathds{1}_{n_i\notin4\ZZ})(n-n_i)+\frac{n_i-p_i}{3}+\mathds{1}_{p_i>8}(p_i-8)\right),\]
where $(n_1,\dots,n_r)$ with $\sum_{i=1}^rn_i=n$ is chosen to be
\[
(n_1,\dots,n_r)=
\begin{cases}
(\underbrace{\lfloor{n/r}\rfloor,\dots,\lfloor{n/r}\rfloor}_{r-p/2},\underbrace{\lfloor{n/r}\rfloor+2,\dots,\lfloor{n/r}\rfloor+2}_{p/2}) & \text{if }p\in[0,r)\cap2\ZZ,\\
(\underbrace{\lfloor{n/r}\rfloor,\dots,\lfloor{n/r}\rfloor}_{r-(p+1)/2},\lfloor{n/r}\rfloor+1,\underbrace{\lfloor{n/r}\rfloor+2,\dots,\lfloor{n/r}\rfloor+2}_{(p-1)/2}) & \text{if }p\in[0,r)\setminus2\ZZ,\\
(\underbrace{\lfloor{n/r}\rfloor-1,\dots,\lfloor{n/r}\rfloor-1}_{r-p/2},\underbrace{\lfloor{n/r}\rfloor+1,\dots,\lfloor{n/r}\rfloor+1}_{p/2}) & \text{if }p\in[r,2r)\cap2\ZZ,\\
(\underbrace{\lfloor{n/r}\rfloor-1,\dots,\lfloor{n/r}\rfloor-1}_{r-(p+1)/2},\lfloor{n/r}\rfloor,\underbrace{\lfloor{n/r}\rfloor+1,\dots,\lfloor{n/r}\rfloor+1}_{(p-1)/2}) & \text{if }p\in[r,2r)\setminus2\ZZ,
\end{cases}
\]
and $p_i$ is the remainder of $n_i$ divided by $12$ for each $i\in[r]$.
\end{theorem}

In particular, when $n$ divides nicely,~\Cref{triangle_even} gives a simple and compact formula for $T_{2r}(n)$.
\begin{corollary}
\label{good_divisibility}
Let $r\geq3$ be a fixed integer. If $n\in\NN$ is sufficiently large and divisible by $12r$, then
\[T_{2r}(n)=\binom{r}{3}\left(\frac{n}{r}\right)^3+\frac{(r-1)n^2}{r}+\frac{n}{3}.\]
\end{corollary}

Our next result is the counterpart of~\Cref{triangle_even} for $k\geq4$. Throughout this paper, whenever we write a function, we assume that all its variables are non-negative integers.
\begin{theorem}
\label{simplex_even}
Let $r\geq k\geq4$ be fixed integers and $n\in\NN$ be sufficiently large. Then
\[S^k_{2r}(n)=\max_{n_1+\dots+n_r=n}f_k(n_1,\dots,n_r),\]
where
\[f_k(n_1,\dots,n_r)=\left(\sum_{\mathcal{I}\in\binom{[r]}{k}}\prod_{i\in\mathcal{I}}n_i\right)
+\left(\sum_{1\leq\ell\leq\lfloor{k/2}\rfloor}\sum_{\mathcal{J}\in\binom{[r]}{\ell}}\left(\prod_{j\in\mathcal{J}}(n_j-\mathds{1}_{n_j\notin4\ZZ})\right)\left(\sum_{\mathcal{I}\in\binom{[r]\setminus\mathcal{J}}{k-2\ell}}\prod_{i\in\mathcal{I}}n_i\right)\right).\]
\end{theorem}

One can readily verify that $\max_{n_1+\dots+n_r=n}f_k(n_1,\dots,n_r)$ must be attained by some $(n_1,\dots,n_r)$ satisfying $|n_i-n/r|=O(1)$ for all $i\in[r]$. However, we make no attempt here to identify the exact maximizer, as it appears to depend intricately on the values of $r$ and $k$. Combining~\Cref{triangle_even,simplex_even}, we determine $S^k_{2r}(n)$ up to a constant in the lower-order term for all $r\geq k\geq 3$. Note that this improves upon~\Cref{simplex} in the even-dimensional case.
\begin{corollary}
\label{error_term}
Let $r\geq k\geq 3$ be fixed integers. Then
\[S^k_{2r}(n)=\binom{r}{k}\left(\frac{n}{r}\right)^k+\Theta(n^{k-1}).\]
\end{corollary}

The main tool for obtaining the precise estimates in~\Cref{triangle_even,simplex_even} is the following stability result, which describes the structure of any almost-extremal point set. We say that two circles are orthogonal if the affine spaces spanned by them are orthogonal.
\begin{theorem}
\label{simplex_stability}
Let $r\geq k\geq 3$ be fixed integers. If a set $X\subseteq\RR^{2r}$ of $n$ points spans $S^k_{2r}(n)-o(n^k)$ regular $(k-1)$-simplices, then one can find a partition $X=A_0\dot\cup A_1\dot\cup\dots\dot\cup A_r$ with $|A_0|=o(n)$ and $|A_i|=n/r-o(n)$ for every $i\in[r]$, and pairwise orthogonal circles $C_1,\dots,C_r\subseteq\RR^{2r}$, such that
\begin{enumerate}
    \item $A_i\subseteq C_i$ for all $i\in[r]$;
    \item $C_1,\dots,C_r$ have the same center and the same radius.
\end{enumerate}
\end{theorem}

The proof of~\Cref{simplex_stability} proceeds in two steps. First, we identify certain forbidden configurations in the point set using linear algebra. Then we apply techniques from hypergraph Tur\'an theory to analyze its structure.

The paper is organized as follows. In~\Cref{sec:construction} we introduce basic geometric notions and present the lower bound constructions for $S^k_d(n)$ needed for~\Cref{simplex,triangle_even,simplex_even}. In~\Cref{sec:preliminary}, we collect known lemmas and establish additional ones concerning hypergraph theory and the orthogonality of affine spaces. In~\Cref{sec:warmup} we give the proof of~\Cref{simplex}. In~\Cref{sec:stability} we establish~\Cref{simplex_stability}. In~\Cref{sec:number} we prove~\Cref{triangle_even,simplex_even} and derive~\Cref{error_term}.~\Cref{sec:concluding} contains concluding remarks and discusses further consequences of our method.

\section{Lower bounds: the Lenz construction}
\label{sec:construction}
In this section we present the Lenz construction of $n$ points in $\RR^d$, which was initially used to give a tight lower bound on the number of unit distances in $\RR^d$ for even $d\geq4$ (see, e.g.,~\cite{KS09}). Here we use refinements of this construction to obtain lower bounds on the number of regular simplices, i.e., the lower bounds for~\Cref{simplex,triangle_even,simplex_even}.

We begin by introducing some basic notions used in the construction and throughout the paper.

\subsection{Basic notions}
We shall give the definitions of affine spaces, spheres, and orthogonality.

We say that a non-empty subset $\mathcal{A}\subseteq\mathbb{R}^d$ is an affine space if there exists $x\in\RR^d$ and a linear subspace $\mathcal{S}\subseteq\mathbb{R}^d$, such that $\mathcal{A}=x+\mathcal{S}:=\{x+s:\,s\in\mathcal{S}\}$. Such an $\mathcal{S}$ is called the associated linear space of $\mathcal{A}$. Observe that if $\mathcal{A}$ is an affine space with associated linear space $\mathcal{S}$, then $\mathcal{S}=\mathcal{A}-a:=\{a'-a:\,a'\in\mathcal{A}\}$ for any $a\in\mathcal{A}$. Hence, every affine space has a unique associated linear space.

The dimension of an affine space $\mathcal{A}\subseteq\mathbb{R}^d$, denoted by $\dim\mathcal{A}$, is defined as the dimension of its associated linear space. Two affine spaces are said to be orthogonal if their associated linear spaces are orthogonal. It follows from the definition that if $\mathcal{A}_1,\dots,\mathcal{A}_r\subseteq\mathbb{R}^d$ are pairwise orthogonal affine spaces, then $\sum_{i\in[r]}\dim\mathcal{A}_i\leq \dim\RR^d=d$.

Given a non-empty subset $A\subseteq\RR^d$, let $\mathrm{Aff}(A)$ denote the affine space spanned by $A$. That is,
\[\mathrm{Aff}(A):=a+\mathrm{Span}(A-a)\quad\text{for any $a\in A$,}\]
where $\mathrm{Span}(A-a)$ represents the linear subspace of $\RR^d$ spanned by $A-a$. For further definitions concerning affine spaces we refer the reader to the book by Roman~\cite{RS05}.

For $d\in\mathbb{N}$, a $(d-1)$-dimensional sphere is defined as the set of all points in a $d$-dimensional affine space $\mathcal{A}$ at some distance $\gamma>0$ from a given point $c\in \mathcal{A}$. The point $c$ and the distance $\gamma$ are referred to as the center and the radius of the sphere, respectively. For example, a set of two points is a $0$-dimensional sphere, and a circle is a $1$-dimensional sphere. A sphere is called unit if its radius is equal to $1$. Two spheres are orthogonal if the affine spaces spanned by them are orthogonal.

\subsection{Lenz construction for even dimensions}
\label{lb_even}
Let $r\geq k\geq3$ be integers. For $n_1,\dots,n_r\in\ZZ_{\geq0}$, we define the function
\begin{align*}
f_k(n_1,\dots,n_r):=\left(\sum_{\mathcal{I}\in\binom{[r]}{k}}\prod_{i\in\mathcal{I}}n_i\right)
&+\left(\sum_{1\leq\ell\leq\lfloor{k/2}\rfloor}\sum_{\mathcal{J}\in\binom{[r]}{\ell}}\left(\prod_{j\in\mathcal{J}}(n_j-\mathds{1}_{n_j\notin4\ZZ})\right)\left(\sum_{\mathcal{I}\in\binom{[r]\setminus\mathcal{J}}{k-2\ell}}\prod_{i\in\mathcal{I}}n_i\right)\right)\\
&+\mathds{1}_{k=3}\left(\sum_{i\in[r]}\frac{n_i-p_i}{3}+\mathds{1}_{p_i>8}(p_i-8)\right),
\end{align*}
where $p_i$ is the remainder of $n_i$ divided by $12$ for each $i\in[r]$. We will show by construction that
\[S^k_{2r}(n)\geq\max_{n_1+\dots+n_r=n}f_k(n_1,\dots,n_r).\]

Let $C_1,\dots,C_r\subseteq\RR^{2r}$ be pairwise orthogonal unit circles with a common center. Such circles can always be picked. For example, we can let $C_i$ be a unit circle in \[\underbrace{\{0\}\times\dots\times\{0\}}_{2(i-1)}\times\RR^2\times\underbrace{\{0\}\times\dots\times\{0\}}_{2r-2i}\] centered at the origin $(0,\dots,0)$, for all $i\in[r]$.

Let $n_1,\dots,n_r\in\ZZ_{\geq0}$ with $n_1+\dots+n_r=n$ maximizing $f(n_1,\dots,n_r)$, and for each $i\in[r]$ let $p_i$ be the remainder of $n_i$ divided by $12$. Let $F=\{v_1,\dots,v_{12}\}$ denote a configuration of $12$ points arranged on a unit circle to form a regular dodecagon, where $v_1,\dots,v_{12}$ are labeled clockwise, see~\Cref{fig1} for an illustration. Note that every point in $F$ is contained in an inscribed square of side length $\sqrt{2}$ and an equilateral triangle of side length $\sqrt{3}$.

\begin{figure}[h!]
\centering
\begin{tikzpicture}[scale=2/5]

\def\radius{5}
\def\n{12}

% Place 12 points on the circle
    \node[circle, fill=black, inner sep=1pt] (P0) at ({\radius*cos(360*0/\n)}, {\radius*sin(360*0/\n)}) {};

    \node[circle, fill=black, inner sep=1pt] (P1) at ({\radius*cos(360*1/\n)}, {\radius*sin(360*1/\n)}) {};

    \node[circle, fill=black, inner sep=1pt] (P2) at ({\radius*cos(360*2/\n)}, {\radius*sin(360*2/\n)}) {};

    \node[circle, fill=black, inner sep=1pt] (P3) at ({\radius*cos(360*3/\n)}, {\radius*sin(360*3/\n)}) {};

    \node[circle, fill=black, inner sep=1pt] (P4) at ({\radius*cos(360*4/\n)}, {\radius*sin(360*4/\n)}) {};

    \node[circle, fill=black, inner sep=1pt] (P5) at ({\radius*cos(360*5/\n)}, {\radius*sin(360*5/\n)}) {};

    \node[circle, fill=black, inner sep=1pt] (P6) at ({\radius*cos(360*6/\n)}, {\radius*sin(360*6/\n)}) {};
    
    \node[circle, fill=black, inner sep=1pt] (P7) at ({\radius*cos(360*7/\n)}, {\radius*sin(360*7/\n)}) {};

    \node[circle, fill=black, inner sep=1pt] (P8) at ({\radius*cos(360*8/\n)}, {\radius*sin(360*8/\n)}) {};

    \node[circle, fill=black, inner sep=1pt] (P9) at ({\radius*cos(360*9/\n)}, {\radius*sin(360*9/\n)}) {};

    \node[circle, fill=black, inner sep=1pt] (P10) at ({\radius*cos(360*10/\n)}, {\radius*sin(360*10/\n)}) {};

    \node[circle, fill=black, inner sep=1pt] (P11) at ({\radius*cos(360*11/\n)}, {\radius*sin(360*11/\n)}) {};
    
% Optional circle outline
\draw[thick, gray] (0,0) circle (\radius);

% === Squares ===
% Square 1: every 3rd point starting from 0 (0,3,6,9)
\draw[blue, opacity=0.5] (P0) -- (P3) -- (P6) -- (P9) -- (P0);

% Square 2: starting from 1 (1,4,7,10)
\draw[blue, opacity=0.5] (P1) -- (P4) -- (P7) -- (P10) -- (P1);

% Square 3: starting from 2 (2,5,8,11)
\draw[blue, opacity=0.5] (P2) -- (P5) -- (P8) -- (P11) -- (P2);

% === Equilateral Triangles ===
% Triangle 1: every 4th point from 0 (0,4,8)
\draw[red, opacity=0.5] (P0) -- (P4) -- (P8) -- (P0);

% Triangle 2: from 1 (1,5,9)
\draw[red, opacity=0.5] (P1) -- (P5) -- (P9) -- (P1);

% Triangle 3: from 2 (2,6,10)
\draw[red, opacity=0.5] (P2) -- (P6) -- (P10) -- (P2);

% Triangle 4: from 3 (3,7,11)
\draw[red, opacity=0.5] (P3) -- (P7) -- (P11) -- (P3);

% Label the points
    \filldraw[black] (P0) circle (1pt) node[right] {$v_{3}$};

    \filldraw[black] (P1) circle (1pt) node[right] {$v_{2}$};

    \filldraw[black] (P2) circle (1pt) node[above] {$v_{1}$};

    \filldraw[black] (P3) circle (1pt) node[above] {$v_{12}$};

    \filldraw[black] (P4) circle (1pt) node[above] {$v_{11}$};

    \filldraw[black] (P5) circle (1pt) node[left] {$v_{10}$};

    \filldraw[black] (P6) circle (1pt) node[left] {$v_{9}$};

    \filldraw[black] (P7) circle (1pt) node[left] {$v_{8}$};

    \filldraw[black] (P8) circle (1pt) node[below] {$v_{7}$};

    \filldraw[black] (P9) circle (1pt) node[below] {$v_{6}$};

    \filldraw[black] (P10) circle (1pt) node[below] {$v_{5}$};

    \filldraw[black] (P11) circle (1pt) node[right] {$v_{4}$};
    
\end{tikzpicture}
\caption{The point configuration of $F$ on a unit circle, where the blue segments are of length $\sqrt{2}$ and the red segments are of length $\sqrt{3}$.}
\label{fig1}
\end{figure}
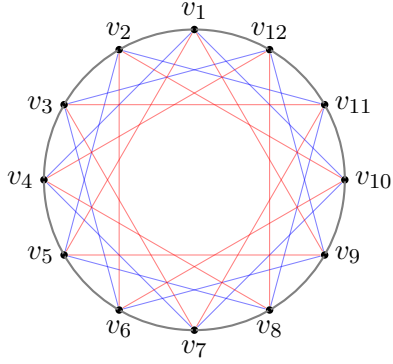

For any $i\in[r]$, we place $n_i$ points on the circle $C_i$ in the following way. Since $n_i-p_i$ is a multiple of $12$, we can place the first $n_i-p_i$ points on $C_i$ by taking $(n_i-p_i)/12$ pairwise disjoint copies of $F$. For the remaining $p_i$ points, we place them consecutively in a new copy of $F$, disjoint from the previous ones, following the order $v_1,v_4,v_7,v_{10},v_2,v_5,v_8,v_{11},v_3,v_6,v_9,v_{12}$.

We now count the number of regular $(k-1)$-simplices spanned by these $n$ points. There are three possible types of simplices, each corresponding to one of the three summands in $f_k(n_1,\dots,n_r)$:
\begin{enumerate}
    \item regular $(k-1)$-simplices $\Delta_1$ with $|\Delta_1\cap C_i|\leq1$ for all $i\in[r]$;
    \item regular $(k-1)$-simplices $\Delta_2$ with $|\Delta_2\cap C_i|=2$ for some $i\in[r]$ and $|\Delta_2\cap C_i|\leq2$ for others;
    \item regular $(k-1)$-simplices $\Delta_3$ containing at least three vertices on the same circle.
\end{enumerate}
Since $C_1,\dots,C_r$ are pairwise orthogonal unit circles, the distance between any two points from different circles is $\sqrt{2}$. Thus, $\Delta_1$ and $\Delta_2$ both have side length $\sqrt{2}$, while $\Delta_3$, if exists, has side length $\sqrt{3}$. The number of $\Delta_1$ is clearly $\sum_{\mathcal{I}\in\binom{[r]}{k}}\prod_{i\in\mathcal{I}}n_i$.

To enumerate the second-type simplices, we need to count the pairs of points on the same circle at distance $\sqrt{2}$ apart, referred to as ``good'' pairs. For any $i\in[r]$, observe that every good pair on $C_i$ consists of either two points from the first $n_i-p_i$ points, or two points from the remaining $p_i$ points. For the first $n_i-p_i$ points, since each of them is at distance $\sqrt{2}$ from two other points, we have $n_i-p_i$ good pairs. For the remaining $p_i$ points, due to the placement scheme, there are $p_i$ good pairs if $p_i$ is divisible by $4$, otherwise we have $p_i-1$ good pairs. Hence, the number of good pairs on $C_i$ is equal to $n_i-\mathds{1}_{p_i\notin4\ZZ}=n_i-\mathds{1}_{n_i\notin4\ZZ}$. Then for $1\leq\ell\leq\lfloor{k/2}\rfloor$, the number of $\Delta_2$ containing exactly $\ell$ good pairs, each pair on a different circle, is equal to
\[\sum_{\mathcal{J}\in\binom{[r]}{\ell}}\left(\prod_{j\in\mathcal{J}}(n_j-\mathds{1}_{n_j\notin4\ZZ})\right)\left(\sum_{\mathcal{I}\in\binom{[r]\setminus\mathcal{J}}{k-2\ell}}\prod_{i\in\mathcal{I}}n_i\right).\]
Accordingly, the total number of $\Delta_2$ is equal to
\[\sum_{1\leq\ell\leq\lfloor{k/2}\rfloor}\sum_{\mathcal{J}\in\binom{[r]}{\ell}}\left(\prod_{j\in\mathcal{J}}(n_j-\mathds{1}_{n_j\notin4\ZZ})\right)\left(\sum_{\mathcal{I}\in\binom{[r]\setminus\mathcal{J}}{k-2\ell}}\prod_{i\in\mathcal{I}}n_i\right).\]

For the number of $\Delta_3$, since there is no third-type simplex when $k\geq4$, we only consider the case when $k=3$. In this case, $\Delta_3$ is an equilateral triangle with three vertices on the same circle. For each circle $C_i$, any equilateral triangle on $C_i$ is formed by either three points from the first $n_i-p_i$ points or three points from the remaining $p_i$ points. Among the first $n_i-p_i$ points, since each of them is contained in one inscribed equilateral triangle, we have $(n_{i}-p_i)/3$ third-type simplices. Among the remaining $p_i$ points, there are $p_i-8$ third-type simplices if $p_i>8$, otherwise we have no third-type simplices. Therefore, the number of $\Delta_3$ in our construction is at least
\[\mathds{1}_{k=3}\left(\sum_{i\in[r]}\frac{n_i-p_i}{3}+\mathds{1}_{p_i>8}(p_i-8)\right).\]

In total, the number of regular $(k-1)$-simplices given by this construction is at least
\begin{align*}
\left(\sum_{\mathcal{I}\in\binom{[r]}{k}}\prod_{i\in\mathcal{I}}n_i\right)
&+\left(\sum_{1\leq\ell\leq\lfloor{k/2}\rfloor}\sum_{\mathcal{J}\in\binom{[r]}{\ell}}\left(\prod_{j\in\mathcal{J}}(n_j-\mathds{1}_{n_j\notin4\ZZ})\right)\left(\sum_{\mathcal{I}\in\binom{[r]\setminus\mathcal{J}}{k-2\ell}}\prod_{i\in\mathcal{I}}n_i\right)\right)\\
&+\mathds{1}_{k=3}\left(\sum_{i\in[r]}\frac{n_i-p_i}{3}+\mathds{1}_{p_i>8}(p_i-8)\right),
\end{align*}
consequently,
\begin{equation}
\label{lower_bound_even}
S^k_{2r}(n)\geq\max_{n_1+\dots+n_r=n}f_k(n_1,\dots,n_r).
\end{equation}

\subsection{Lenz construction for odd dimensions}
\label{lb_odd}
Let $r\geq k\geq3$ be integers. Let $C_1,\dots,C_{r}\subseteq\RR^{2r+1}$ be pairwise orthogonal unit spheres with a common center, where $C_1,\dots,C_{r-1}$ are circles and $C_r$ is a $2$-dimensional sphere. 

We put $n$ points on $C_1,\dots,C_r$ as evenly as possible, i.e., each $C_i$ receives $\lfloor{n/r}\rfloor$ or $\lceil{n/r}\rceil$ points. Since the distance between any two points on different spheres is $\sqrt{2}$, the number of regular $(k-1)$-simplices spanned by $k$ points from pairwise distinct spheres is $\binom{r}{k}\left(\frac{n}{r}\right)^k-O(n^{k-1})$.

Moreover, Erd\H{o}s, Hickerson, and Pach~\cite{EHP89} proved that there is a configuration of $m$ points on a $2$-dimensional unit sphere, such that the distance $\sqrt{2}$ occurs $\Omega(m^{4/3})$ times. The points on $C_r$ are arranged according to this configuration, which yields $\Omega(n^{4/3})$ pairs of points on $C_r$ at distance $\sqrt{2}$ apart. Note that every such pair forms a regular $(k-1)$-simplex with any $k-2$ points from pairwise distinct circles among $C_1,\dots,C_{r-1}$. Hence, there are at least $\Omega(n^{k-2/3})$ regular $(k-1)$-simplices spanned by $2$ points from $C_r$ and $k-2$ points from $C_1,\dots,C_{r-1}$. Overall we obtain that
\begin{equation}
\label{lower_bound_odd}
S^k_{2r+1}(n)\geq \binom{r}{k}\left(\frac{n}{r}\right)^k+\Omega(n^{k-2/3}).   
\end{equation}

\section{Preliminary results and lemmas}
\label{sec:preliminary}
\subsection{Tur\'an-type results for hypergraphs}
Given an integer $k\geq2$, a $k$-uniform hypergraph $H$ is an ordered pair $(V(H),E(H))$, where $V(H)$ is the set of vertices and $E(H)\subseteq\binom{V(H)}{k}$ is the set of hyperedges. Let $v(H)$ and $e(H)$ denote the numbers of vertices and hyperedges in $H$, respectively.

A $k$-uniform hypergraph $H$ is called $r$-partite for some integer $r\geq k$ if there exists a partition $V(H)=V_1\dot\cup V_2\dot\cup\dots\dot\cup V_r$, such that every hyperedge in $H$ has at most one vertex from each part $V_i$. 

For a $k$-uniform hypergraph $H$, the extremal number $\mathrm{ex}(n,H)$ is the largest number of hyperedges in an $n$-vertex $k$-uniform hypergraph containing no copy of $H$. Given an integer $t\geq1$, the $t$-blowup of $H$, denoted by $H(t)$, is a $k$-uniform hypergraph obtained from $H$ by replacing every vertex $v\in V(H)$ with an independent set $I_v$ of size $t$ and replacing every hyperedge $e=\{v_1,\dots,v_k\}\in E(H)$ with a complete $k$-partite $k$-uniform hypergraph spanned
by $I_{v_1}\dot\cup\dots\dot\cup I_{v_k}$. The independent set $I_v$ is referred to as the blowup of the vertex $v$.

\begin{definition}
Given integers $r\geq k\geq3$, $H^{(k)}_{r+1}$ is defined as a $k$-uniform hypergraph obtained from a complete graph $K_{r+1}$ by enlarging each edge with $k-2$ new vertices.
\end{definition}
Note that by definition $H^{(k)}_{r+1}$ contains $(r+1)+\binom{r+1}{2}(k-2)$ vertices and $\binom{r+1}{2}$ hyperedges.

\begin{lemma}[Mubayi~\cite{DM06}]
\label{turan}
For any fixed integers $r\geq k\geq3$, $\mathrm{ex}(n,H^{(k)}_{r+1})=\binom{r}{k}\left(\frac{n}{r}\right)^k+o(n^k)$.
\end{lemma}

For any $k$-uniform hypergraph $H$ and $t\in\NN$, it is well-known (see, e.g.,~\cite{PK11}) that
\[\mathrm{ex}(n,H(t))=\mathrm{ex}(n,H)+o(n^k).\]
Thus, by~\Cref{turan} we have
\begin{equation}
\label{blowup}
    \mathrm{ex}(n,H^{(k)}_{r+1}(t))=\mathrm{ex}(n,H^{(k)}_{r+1})+o(n^k)=\binom{r}{k}\left(\frac{n}{r}\right)^k+o(n^k).
\end{equation}

Furthermore, we shall prove the following stability lemma for $H^{(k)}_{r+1}(t)$.

\begin{lemma}
\label{new:stability}
Let $r\geq k\geq3$ and $t\geq1$ be fixed integers. If $G$ is an $n$-vertex $H^{(k)}_{r+1}(t)$-free $k$-uniform hypergraph with $\mathrm{ex}(n,H^{(k)}_{r+1}(t))-o(n^k)$ hyperedges, then $G$ contains an $r$-partite subhypergraph with $\binom{r}{k}(n/r)^k-o(n^k)$ hyperedges.
\end{lemma}

We deduce~\Cref{new:stability} by a standard argument, using a stability lemma for $H^{(k)}_{r+1}$ and the hypergraph removal lemma.

\begin{lemma}[Pikhurko~\cite{OP13}]
\label{stabilityH}
Let $r\geq k\geq3$ be fixed integers. Every $n$-vertex $H^{(k)}_{r+1}$-free $k$-uniform hypergraph with $\mathrm{ex}(n,H^{(k)}_{r+1})-o(n^k)$ hyperedges contains an $r$-partite subhypergraph with $\binom{r}{k}(n/r)^k-o(n^k)$ hyperedges.
\end{lemma}

\begin{lemma}[R\"{o}dl--Nagle--Skokan--Schacht--Kohayakawa~\cite{RN05}]
\label{removal}
Let $k\geq2$ be an integer and $H$ be a fixed $k$-uniform hypergraph. If an $n$-vertex $k$-uniform hypergraph $G$ contains $o(n^{v(H)})$ copies of $H$, then one can remove all the copies of $H$ in $G$ by deleting $o(n^k)$ hyperedges.
\end{lemma}

\begin{proof}[Proof of~\Cref{new:stability}]
Let $G$ be an $n$-vertex $H^{(k)}_{r+1}(t)$-free $k$-uniform hypergraph with $\mathrm{ex}(n,H^{(k)}_{r+1}(t))-o(n^k)$ hyperedges. By~\eqref{blowup} we have $e(G)=\mathrm{ex}(n,H^{(k)}_{r+1})-o(n^k)$. Note that the hypergraph $G$ cannot contain $\Omega(n^{v(H^{(k)}_{r+1})})$ copies of $H^{(k)}_{r+1}$, as it otherwise were to contain a copy of $H^{(k)}_{r+1}(t)$. Now that the number of copies of $H^{(k)}_{r+1}$ in $G$ is $o(n^{v(H^{(k)}_{r+1})})$, by~\Cref{removal} there exists an $H^{(k)}_{r+1}$-free subhypergraph $G'$ of $G$ with 
\[
e(G')=e(G)-o(n^k)=\mathrm{ex}(n,H^{(k)}_{r+1})-o(n^k)
\]
hyperedges. The statement then follows from~\Cref{stabilityH} applied to $G'$.
\end{proof}

The following fact has been used implicitly in the literature (see, e.g.,~\cite{DM06,OP13}), and we include a proof here for completeness.
\begin{lemma}
\label{parts}
Let $r\geq k\geq2$ be fixed integers. If an $n$-vertex $r$-partite $k$-uniform hypergraph has $\binom{r}{k}(n/r)^k-o(n^k)$ hyperedges, then each part must be of size $n/r-o(n)$.
\end{lemma}
\begin{proof}[Proof of~\Cref{parts}]
Let $G$ be such a hypergraph and let $n_1\leq\dots\leq n_r$ denote the part sizes of $G$. For each $i\in[r]$ let $\alpha_i=n_i/n$ be the scaled part size. Note that $\alpha_1+\dots+\alpha_r=1$. Then
\[e(G)\leq\sum_{\mathcal{I}\in\binom{[r]}{k}}\prod_{i\in\mathcal{I}}n_i=n^k\sum_{\mathcal{I}\in\binom{[r]}{k}}\prod_{i\in\mathcal{I}}\alpha_i\]
By Maclaurin's inequality, see~\cite{BP03},
\begin{equation}
\label{MacLaurin}
\sum_{\mathcal{I}\in\binom{[r]}{k}}\prod_{i\in\mathcal{I}}\alpha_i\leq\binom{r}{k}\left(\frac{1}{r}\right)^k
\end{equation}
holds for all $\alpha_1,\dots,\alpha_r\geq0$ with $\alpha_1+\dots+\alpha_r=1$. Moreover, the equality in~\eqref{MacLaurin} holds if and only if $\alpha_1=\dots=\alpha_r=1/r$.

Suppose for contradiction that $n_r\geq n/r+\delta n$ with some constant $\delta>0$. Since $\alpha_r\geq 1/r+\delta$, the inequality in~\eqref{MacLaurin} is strict. Thus, there exists some constant $\eps>0$ such that $e(G)\leq\binom{r}{k}(n/r)^k-\eps n^k$, a contradiction.
\end{proof}

\subsection{Orthogonality of affine spaces}
Let $||\cdot||$ denote the Euclidean norm. For $P,Q\subseteq\RR^d$, we write $P\to Q$ if for every $p\in P$, $||p-q||=||p-q'||$ holds for all $q,q'\in Q$. In other words, $P\to Q$ if every point in $P$ is equidistant to all points in $Q$. We write $P\leftrightarrow Q$ if $P\to Q$ and $Q\to P$. Recall that $\mathrm{Aff}(P)$ denotes the affine space spanned by $P$.
\begin{lemma}
\label{affine}
Let $d\in\NN$ and let $P,Q\subseteq\RR^d$ be non-empty subsets with $|Q|\geq2$ such that $P\to Q$. Then $\mathrm{Aff}(P)$ and $\mathrm{Aff}(Q)$ are orthogonal. Moreover, $Q$ is contained in a sphere in $\mathrm{Aff}(Q)$.
\end{lemma}
\begin{proof}
For the first part of~\Cref{affine}, if $|P|=1$, then $\mathrm{Aff}(P)$ and $\mathrm{Aff}(Q)$ are clearly orthogonal as $\mathrm{Aff}(P)$ consists of a single point. If $|P|\geq2$, let $\{p_1,p_2,\dots\}\subseteq P$ and $\{q_1,q_2,\dots\}\subseteq Q$ be finite subsets that span the affine spaces $\mathrm{Aff}(P)$ and $\mathrm{Aff}(Q)$, respectively. Then we have \[\mathrm{Aff}(P)=p_1+\mathrm{Span}(p_2-p_1,\dots)\quad\text{and}\quad\mathrm{Aff}(Q)=q_1+\mathrm{Span}(q_2-q_1,\dots).\] To show that $\mathrm{Aff}(P)$ and $\mathrm{Aff}(Q)$ are orthogonal, it suffices to check $\langle{p_i-p_1,q_j-q_1}\rangle=0$ for $i,j\geq2$. Due to $P\to Q$, we have $\lVert{p_1-q_1}\rVert^2=\lVert{p_1-q_j}\rVert^2$ and $\lVert{p_i-q_1}\rVert^2=\lVert{p_i-q_j}\rVert^2$. Then
\begin{align*}
\langle{p_i-p_1,q_j-q_1}\rangle&=\langle{p_i,q_j}\rangle-\langle{p_i,q_1}\rangle-\langle{p_1,q_j}\rangle+\langle{p_1,q_1}\rangle\\
&=\frac{\lVert{p_i}\rVert^2+\lVert{q_j}\rVert^2-\lVert{p_i-q_j}\rVert^2}{2}-\frac{\lVert{p_i}\rVert^2+\lVert{q_1}\rVert^2-\lVert{p_i-q_1}\rVert^2}{2}\\
&-\frac{\lVert{p_1}\rVert^2+\lVert{q_j}\rVert^2-\lVert{p_1-q_j}\rVert^2}{2}+\frac{\lVert{p_1}\rVert^2+\lVert{q_1}\rVert^2-\lVert{p_1-q_1}\rVert^2}{2}=0.
\end{align*}

Now we prove the second part of~\Cref{affine}. Provided that $\mathrm{Aff}(P)$ and $\mathrm{Aff}(Q)$ are orthogonal, the projection of $\mathrm{Aff}(P)$ onto $\mathrm{Aff}(Q)$ has dimension $0$, which is a point, denoted by $c$. Since each point of $P$ is equidistant to all points in $Q$, its projection $c$ is also equidistant to all points in $Q$. Thus the points in $Q$ all lie on a sphere in $\mathrm{Aff}(Q)$ centered at $c$.
\end{proof}

The following lemma is a direct consequence of~\Cref{affine}.
\begin{lemma}
\label{witzig}
Let $d,k\in\NN$ with $k\geq3$. Let $A_1,\dots,A_k\subseteq \RR^d$ with $\lvert{A_i}\rvert\geq3$ for all $i\in[k]$. If every $k$-tuple $(a_1,\dots,a_k)\in \prod_{i\in[k]}A_i$ forms a regular $(k-1)$-simplex, then $\mathrm{Aff}(A_1),\dots,\mathrm{Aff}(A_k)$ are pairwise orthogonal affine spaces of dimension at least $2$. Furthermore, $A_1,\dots,A_k$ are subsets of $k$ spheres in $\mathrm{Aff}(A_1),\dots,\mathrm{Aff}(A_k)$, respectively.
\end{lemma}
\begin{proof}[Proof of~\Cref{witzig}]
First observe that $||a_1-a_2||=||a_1-a_2'||$ holds for every $a_1\in A_1$ and $a_2,a_2'\in A_2$. Otherwise for any $(a_3,\dots,a_k)\in\prod_{i=3}^{k}A_i$, either $(a_1,a_2,a_3,\dots,a_k)$ or $(a_1,a_2',a_3,\dots,a_k)$ does not form a regular $(k-1)$-simplex, a contradiction. Thus, $A_1\to A_2$. Due to symmetry, we conclude that $A_i\leftrightarrow A_j$ for all distinct $i,j\in[k]$.

Then by~\Cref{affine}, the affine spaces $\mathrm{Aff}(A_1),\dots,\mathrm{Aff}(A_k)$ are pairwise orthogonal and $A_1,\dots,A_k$ are subsets of $k$ spheres in $\mathrm{Aff}(A_1),\dots,\mathrm{Aff}(A_k)$, respectively. Moreover, since each $A_i$ is a subset of a sphere and $|A_i|\geq3$, the points in $A_i$ are non-collinear, which implies that the dimension of $\mathrm{Aff}(A_i)$ is at least $2$.
\end{proof}

We also need the following lemma concerning the geometric properties of a collection of pairwise orthogonal spheres.
\begin{lemma}
\label{spheres_distance_center}
Let $d,k\in\NN$ with $k\geq3$. Let $C_1,\dots,C_k\subseteq\mathbb{R}^d$ be a collection of pairwise orthogonal spheres. For each $i\in[k]$ let $A_i\subseteq C_i$ be such that $\mathrm{Aff}(A_i)=\mathrm{Aff}(C_i)$.
\begin{enumerate}[(i)]
    \item If every $k$-tuple in $\prod_{i\in[k]}A_i$ forms a regular $(k-1)$-simplex, then every $k$-tuple in $\prod_{i\in[k]}C_i$ forms a regular $(k-1)$-simplex.
    \item If every $k$-tuple in $\prod_{i\in[k]}C_i$ forms a regular $(k-1)$-simplex and $\sum_{i\in[k]}\dim\mathrm{Aff}(C_i)=d$, then $C_1,\dots,C_k$ have the same center.
\end{enumerate}
\end{lemma}
\begin{proof}[Proof of~\Cref{spheres_distance_center}]
Let $c_1,\dots,c_r$ denote the centers of $C_1,\dots,C_r$, respectively.

We first prove item (i). Fix any distinct $i,j\in[k]$ and let $(x,y)\in C_i\times C_j$. Let $(a_i,a_j)\in A_i\times A_j$. Since every $k$-tuple in $\prod_{\ell\in[k]}A_\ell$ forms a regular $(k-1)$-simplex, it suffices to show that $||x-y||=||a_i-a_j||$. Note that every $k$-tuple in $\prod_{\ell\in[k]}A_\ell$ forming a regular $(k-1)$-simplex also implies that $A_i\leftrightarrow A_j$, from which it follows that the projection of $\mathrm{Aff}(A_i)=\mathrm{Aff}(C_i)$ onto $\mathrm{Aff}(A_j)=\mathrm{Aff}(C_j)$ is the center of a sphere in $\mathrm{Aff}(A_j)$ containing $A_j$. Due to $\mathrm{Aff}(A_j)=\mathrm{Aff}(C_j)$, this sphere is exactly $C_j$. Hence, the projection of $\mathrm{Aff}(C_i)$ onto $\mathrm{Aff}(C_j)$ is the center $c_j$ of $C_j$. Similarly, the projection of $\mathrm{Aff}(C_j)$ onto $\mathrm{Aff}(C_i)$ is the center $c_i$ of $C_i$. Then we have that
\[||x-y||^2=||x-c_i||^2+||c_i-c_j||^2+||y-c_j||^2=||a_i-c_i||^2+||c_i-c_j||^2+||a_j-c_j||^2=||a_i-a_j||^2,\]
which completes the proof of item (i).

For item (ii), since $\mathrm{Aff}(C_1),\dots,\mathrm{Aff}(C_k)$ are pairwise orthogonal and their dimensions sum to $d$, we can equip $\RR^d$ with an orthonormal basis given by the union of $B_1,\dots,B_k$, where each $B_i$ is an orthonormal basis of the linear subspace $\mathrm{Aff}(C_i)-c_i$ and $|B_i|=\dim\mathrm{Aff}(C_i)\geq1$. Under this basis of $\RR^d$, we can assume that
\[\mathrm{Aff}(C_i)=c_i+\underbrace{\{0\}\times\dots\times\{0\}}_{\sum_{j<i}\dim\mathrm{Aff}(C_j)}\times\RR^{\dim\mathrm{Aff}(C_i)}\times\underbrace{\{0\}\times\dots\times\{0\}}_{\sum_{j>i}\dim\mathrm{Aff}(C_j)}\] for all $i\in[k]$. Fix any distinct $i,j\in[k]$. It suffices to show that $c_i=c_j$. Since every $k$-tuple in $\prod_{\ell\in[k]}C_\ell$ forms a regular $(k-1)$-simplex, the projections of $c_i$ and $c_j$ onto $\mathrm{Aff}(C_1)$ are both $c_1$. Consequently, the first $\dim\mathrm{Aff}(C_1)$ coordinates of $c_i$ and $c_j$ must coincide. By the same reasoning, the remaining coordinates of $c_i$ and $c_j$ must also coincide, and therefore $c_i=c_j$.
\end{proof}

\section{The asymptotics of $S_d^k(n)$}
\label{sec:warmup}
In this section, we present a short proof of~\Cref{simplex}, which also serves as a warm-up for the arguments in the next section. Our key observation is the following lemma.

\begin{lemma}
\label{Hfree}
Let $d\geq2k\geq6$ be integers and let $r=\lfloor{d/2}\rfloor$. Let $G$ be a $k$-uniform hypergraph, whose vertices are points in $\mathbb{R}^d$ and whose hyperedges are the regular $(k-1)$-simplices spanned by its vertices. Then $G$ is $H^{(k)}_{r+1}(3)$-free.
\end{lemma}
\begin{proof}[Proof of~\Cref{Hfree}]
Suppose that $G$ contains a copy of $H^{(k)}_{r+1}(3)$ and let \[U_1,\dots,U_{r+1},W_1,\dots,W_{\binom{r+1}{2}(k-2)}\] be the blowups of the vertices, where $U_{1},\dots,U_{r+1}$ correspond to the vertices of the original $K_{r+1}$. Then for any distinct $i,j\in[r+1]$, there exist some $W_{\ell_1},\dots,W_{\ell_{k-2}}$, such that every
\[(u_i,u_j,w_{\ell_1},\dots,w_{\ell_{k-2}})\in U_i\times U_j\times W_{\ell_1}\times\dots\times W_{\ell_{k-2}}\]
appears as a hyperedge in $G$, namely, $(u_i,u_j,w_{\ell_1},\dots,w_{\ell_{k-2}})$ determines a regular $(k-1)$-simplex. Then by~\Cref{witzig}, in particular $\mathrm{Aff}(U_i)$ and $\mathrm{Aff}(U_j)$ are orthogonal affine spaces of dimension at least $2$. Applying the same argument to all distinct $i,j\in[r+1]$, we obtain that $\mathrm{Aff}(U_1),\dots,\mathrm{Aff}(U_{r+1})$ are pairwise orthogonal and all have dimension at least $2$. But then we have
\begin{equation}
\label{dimension}
2(r+1)\leq \sum_{i\in[r+1]}\dim\mathrm{Aff}(U_i)\leq\dim\RR^d=d<2(r+1),  
\end{equation}
a contradiction.
\end{proof}

\begin{proof}[Proof of~\Cref{simplex}]
Let $r=\lfloor{d/2}\rfloor\geq k\geq3$ be fixed integers and $X\subseteq\RR^{d}$ be a set of $n$ points that spans $S^k_d(n)$ regular $(k-1)$-simplices. We define a $k$-uniform hypergraph $G$, whose vertex set is $X$ and whose hyperedges correspond to regular $(k-1)$-simplices spanned by $X$.

By~\Cref{Hfree} we know that $G$ is $H^{(k)}_{r+1}(3)$-free. Then by~\eqref{blowup} we have that \[S_d^k(n)=e(G)\leq\binom{r}{k}\left(\frac{n}{r}\right)^k+o(n^k).\]
The lower bound on $S_d^k(n)$ follows from the construction described in~\Cref{sec:construction}, where we take $r$ pairwise orthogonal unit circles centered at the origin and place $n$ points on them as evenly as possible. This completes the proof.
\end{proof}

\section{The stability property of an almost-extremal point set}
\label{sec:stability}
In this section we prove our stability result.
\begin{proof}[Proof of~\Cref{simplex_stability}]
Let $r\geq k\geq3$ be fixed integers and let $X\subseteq\RR^{2r}$ be a set of $n$ points that spans at least $S^k_{2r}(n)-o(n^k)$ regular $(k-1)$-simplices. Since $S^k_{2r}(n)=\binom{r}{k}(n/r)^k+o(n^k)$ by~\Cref{simplex}, $X$ spans at least $\binom{r}{k}(n/r)^k-o(n^k)$ regular $(k-1)$-simplices. Let $G$ be a $k$-uniform hypergraph on the vertex set $X$, where $E(G)$ consists of all regular $(k-1)$-simplices spanned by $X$.

By~\Cref{Hfree}, $G$ is $H^{(k)}_{r+1}(3)$-free. Moreover, by~\eqref{blowup} we have that \[e(G)\geq\binom{r}{k}\left(\frac{n}{r}\right)^k-o(n^k)\geq\mathrm{ex}(n,H^{(k)}_{r+1}(3))-o(n^k).\]
It then follows from~\Cref{new:stability} that $G$ contains an $r$-partite subhypergraph $M$ with vertex classes $V_1\dot\cup\dots\dot\cup V_r$ and $\binom{r}{k}(n/r)^k-o(n^k)$ hyperedges. Note that by~\Cref{parts} we have $|V_i|=n/r-o(n)$ for all $i\in[r]$.

We shall prove two claims that establish the structure of $X$ described in~\Cref{simplex_stability}.

\begin{claim}
\label{lustig}
For each $i\in[r]$ the following holds.
\begin{enumerate}[(i)]
    \item There exists $A_i\subseteq V_i$ with $|A_i|=n/r-o(n)$, such that $\dim\mathrm{Aff}(A_i)=2$.
    \item There is a circle $C_i$ such that $A_i\subseteq C_i \subseteq\mathrm{Aff}(A_i)$.
\end{enumerate}
\end{claim}
\begin{proof}[Proof of~\Cref{lustig}]
It suffices to prove the statement for the case $i=r$, as the remaining cases follow by the same argument. Fix a sufficiently small constant $\eps>0$. For any distinct $\ell_1,\dots,\ell_{k-1}\in[r-1]$, a $(k-1)$-tuple $(v_{\ell_1},\dots,v_{\ell_{k-1}})\in \prod_{i\in[k-1]}V_{\ell_i}$ is called ``good'' if there are at least $|V_r|-\eps n$ vertices $v_r\in V_r$ such that $v_{\ell_1}\dots v_{\ell_{k-1}}v_r$ is a hyperedge in $M$, otherwise it is ``bad''. The total number of bad tuples is at most $\eps n^{k-1}$, because otherwise \[e(M)\leq\left(\sum_{\mathcal{I}\in\binom{[r]}{k}}\prod_{i\in\mathcal{I}}|V_i|\right)-\eps n^{k-1}\cdot\eps n\leq \binom{r}{k}\left(\frac{n}{r}\right)^k-\eps^2 n^k+o(n^k),\] a contradiction.
So if we take a $(k-1)$-tuple $(v_{\ell_1},\dots,v_{\ell_{k-1}})\in \prod_{i\in[k-1]}V_{\ell_i}$ uniformly at random, the probability that $(v_{\ell_1},\dots,v_{\ell_{k-1}})$ is bad is at most
\[\mathbb{P}_{\text{bad}}\leq\frac{\eps n^{k-1}}{\prod_{i\in[k-1]}|V_{\ell_i}|}\leq r^{k-1}\eps+o(1).\]
Now for each $i\in[r-1]$ we take a subset $V_i'\subseteq V_i$ of size $3$ uniformly at random. Then by the union bound
\[\mathbb{P}\left(\forall\,\mathcal{I}\in\binom{[r-1]}{k-1}:\text{ all the $(k-1)$-tuples in $\prod_{i\in\mathcal{I}}V_i'$ are good}\right)\geq 1- \binom{r-1}{k-1}3^{k-1}\mathbb{P}_{\text{bad}}>0,\]
provided that $\eps$ is sufficiently small. Accordingly, there exists $V_i'\subseteq V_i$ of size $3$ for each $i\in[r-1]$, such that all $(k-1)$-tuples among them are good. Let $A_r\subseteq V_r$ be the common neighborhood of the $(k-1)$-tuples in $\bigcup_{\mathcal{I}\in\binom{[r-1]}{k-1}}\prod_{i\in\mathcal{I}}V_i'$. Then
\begin{align*}
|A_r|\geq |V_r|- \binom{r-1}{k-1}3^{k-1}\eps n\geq\frac{n}{r}-(3r)^{k-1}\eps n-o(n).
\end{align*}
In particular, for any $\mathcal{I}\in\binom{[r-1]}{k-1}$, it holds that every $k$-tuple in $\left(\prod_{i\in\mathcal{I}}V_i'\right)\times A_r$ appears as a hyperedge in $M$, i.e., a regular $(k-1)$-simplex. By~\Cref{witzig} we have that $\mathrm{Aff}(V_1'),\dots,\mathrm{Aff}(V_{r-1}'),\mathrm{Aff}(A_r)$ are pairwise orthogonal affine spaces of dimension at least $2$. Since \[2r\leq\dim\mathrm{Aff}(V_1')+\dots+\dim\mathrm{Aff}(V_{r-1}')+\dim\mathrm{Aff}(A_r)\leq \dim\mathbb{R}^{2r}=2r,\]
it further holds that $\dim\mathrm{Aff}(V_1')=\dots=\dim\mathrm{Aff}(V_{r-1}')=\dim\mathrm{Aff}(A_r)=2$. Moreover, by~\Cref{witzig} the points of $A_r$ all lie on some circle $C_r\subseteq\mathrm{Aff}(A_r)$.

We complete the proof by noting that $\eps$ can be chosen arbitrarily small.
\end{proof}

Let $A_1,\dots,A_r$ and $C_1,\dots,C_r$ be given by~\Cref{lustig}.
\begin{claim}
\label{spheres}
We have that
\begin{enumerate}[(i)]
    \item $C_1,\dots,C_r$ are pairwise orthogonal circles;
    \item $C_1,\dots,C_r$ have the same center and the same radius.
\end{enumerate}
\end{claim}
\begin{proof}[Proof of~\Cref{spheres}]
To prove item (i), it suffices to show that the affine planes $\mathrm{Aff}(C_1)=\mathrm{Aff}(A_1)$ and $\mathrm{Aff}(C_2)=\mathrm{Aff}(A_2)$ are orthogonal, as then by symmetry we have that $\mathrm{Aff}(C_1),\dots,\mathrm{Aff}(C_r)$ are pairwise orthogonal. Since $A_i\subseteq V_i$ with $|A_i|=n/r-o(n)$ for all $i\in[r]$, we have that \[e(M[A_1,\dots,A_r])=e(M)-o(n^{k})=\binom{r}{k}\left(\frac{n}{r}\right)^k-o(n^k),\] where $M[A_1,\dots,A_r]$ denotes the induced subhypergraph of $M$ on $A_1\dot\cup\dots\dot\cup A_r$. In particular, the number of hyperedges with one endpoint in each $A_i$ for $i\in[k]$ is at least $(n/r)^k-o(n^k)$. Let $K^{(k)}_{3,\dots,3}$ denote a complete $k$-partite $k$-uniform hypergraph with each part of size $3$. A classical result of Erd\H{o}s~\cite{Er64} shows that
\[\mathrm{ex}(n,K^{(k)}_{\scriptsize\underbrace{3,\dots,3}_{k}})=o(n^k),\]
which implies that there exists $A_i'\subseteq A_i$ of size $3$ for each $i\in[k]$, such that every $k$-tuple in $\prod_{i\in[k]}A_i'$ forms a hyperedge, i.e., a regular $(k-1)$-simplex. Then by~\Cref{witzig}, $\mathrm{Aff}(A_1')$ and $\mathrm{Aff}(A_{2}')$ are orthogonal affine spaces of dimension at least $2$. Since $\dim\mathrm{Aff}(A_1)=\dim\mathrm{Aff}(A_2)=2$, we have that $\mathrm{Aff}(A_1')=\mathrm{Aff}(A_1)$ and $\mathrm{Aff}(A_2')=\mathrm{Aff}(A_2)$. Thus $\mathrm{Aff}(A_1)$ and $\mathrm{Aff}(A_{2})$ are orthogonal. This proves item (i).

We now prove item (ii).

First, we show that every $r$-tuple in $\prod_{i\in[r]}C_i$ forms a regular $(r-1)$-simplex. In fact, it suffices to prove that for any $\mathcal{I}\in\binom{[r]}{k}$, every $k$-tuple in $\prod_{i\in\mathcal{I}}C_i$ forms a regular $(k-1)$-simplex. Here we only consider the case $\mathcal{I}=[k]$, as the remaining cases follow by symmetry. In the above proof of item (i) we already showed that there exists $A_i'\subseteq A_i\subseteq C_i$ with $\mathrm{Aff}(A_i')=\mathrm{Aff}(A_i)=\mathrm{Aff}(C_i)$ for each $i\in[k]$, such that every $k$-tuple in $\prod_{i\in[k]}A_i'$ forms a regular $(k-1)$-simplex. Then by~\Cref{spheres_distance_center} (i) we can conclude that every $k$-tuple in $\prod_{i\in[k]}C_i$ forms a regular $(k-1)$-simplex. 

Next, we show that $C_1,\dots,C_r$ have the same center. From item (i) we know that $C_1,\dots,C_r$ are pairwise orthogonal circles with $\sum_{i\in[r]}\dim\mathrm{Aff}(C_i)=2r$. Since every $r$-tuple in $\prod_{i\in[r]}C_i$ forms a regular $(r-1)$-simplex, by~\Cref{spheres_distance_center} (ii), $C_1,\dots,C_r$ have the same center.

Finally, we prove that $C_1,\dots,C_r$ have the same radius. For any distinct $i,j,\ell\in[r]$, let $\gamma_i,\gamma_j,\gamma_\ell$ denote the radii of $C_i,C_j,C_\ell$, respectively. Since $C_i,C_j,C_\ell$ are pairwise orthogonal with the same center and every $r$-tuple in $\prod_{i\in[r]}C_i$ forms a regular $(r-1)$-simplex, we have that $\gamma_i^2+\gamma_\ell^2=\gamma_j^2+\gamma_\ell^2$, yielding that $\gamma_i=\gamma_j$. This completes the proof of item (ii).
\end{proof}

Let $A_0\subseteq X$ be the set of points that are not in $\bigcup_{i\in[r]}A_i$. In summary, we have shown that one can find a partition $X=A_0\dot\cup A_1\dot\cup\dots\dot\cup A_r$ with $|A_0|=o(n)$ and $|A_i|=n/r-o(n)$ for every $i\in[r]$ (see~\Cref{lustig} (i)), and pairwise orthogonal circles $C_1,\dots,C_r\subseteq\RR^d$ (see~\Cref{spheres} (i)), such that $A_i\subseteq C_i$ for all $i\in[r]$ (see~\Cref{lustig} (ii)) and $C_1,\dots,C_r$ have the same center and the same radius (see~\Cref{spheres} (ii)). This completes the proof of~\Cref{simplex_stability}.
\end{proof}

\section{The number of regular $(k-1)$-simplices in even dimensions}
\label{sec:number}
In this section we prove~\Cref{triangle_even,simplex_even} and derive~\Cref{error_term} from them. We shall prove~\Cref{simplex_even} first. The proof of~\Cref{triangle_even} then follows with an additional optimization step.

\subsection{The general case}
\begin{proof}[Proof of~\Cref{simplex_even}]
Let $r\geq k\geq3$ be fixed integers. Let $X\subseteq\RR^{2r}$ be a set of $n$ points that attains $S^k_{2r}(n)$, where $n$ is sufficiently large. In order to establish~\Cref{simplex_even}, we shall first apply~\Cref{simplex_stability} to get a coarse structure of $X$, and then further examine this structure using the extremality of $X$.

By~\Cref{simplex_stability} we have a partition $X=A_0\dot\cup A_1\dot\cup\dots\dot\cup A_r$ with $|A_0|=o(n)$ and $|A_i|=n/r-o(n)$ for every $i\in[r]$, and pairwise orthogonal circles $C_1,\dots,C_r\subseteq\RR^{2r}$ with $A_i\subseteq C_i$ for all $i\in[r]$, such that $C_1,\dots,C_r$ have the same center and the same radius $\gamma>0$. In particular, we can assume that $A_{0}=X\setminus(\bigcup_{i\in[r]}C_i)$.

Since $\mathrm{Aff}(C_1),\dots,\mathrm{Aff}(C_r)$ are pairwise orthogonal and intersect at the common center of $C_1,\dots,C_r$, we can further assume without loss of generality that the common center is the origin and \[\mathrm{Aff}(C_i)=\underbrace{\{0\}\times\dots\times\{0\}}_{2(i-1)}\times\RR^2\times\underbrace{\{0\}\times\dots\times\{0\}}_{2r-2i}\] for all $i\in[r]$.

\begin{claim}
\label{ellig}
We have $X\subseteq\bigcup_{i\in[r]}C_{i}$, i.e., $A_0=\emptyset$.
\end{claim}
\begin{proof}[Proof of~\Cref{ellig}]
Suppose there exists $v\in A_0=X\setminus\bigcup_{i\in[r]}C_{i}$. Let $\mathcal{I}(v)\subseteq[r]$ be the set of indices $i\in[r]$, such that the projection of $v$ onto $\mathrm{Aff}(C_i)$ is the center of $C_i$, i.e., the origin. Observe that the distance between $v$ and any point in $\bigcup_{i\in\mathcal{I}(v)}C_i$ is $\sqrt{||v||^2+\gamma^2}$.

Let $N_1$, $N_2$, and $N_3$ denote the number of regular $(k-1)$-simplices spanned by $X$ that
\begin{itemize}
    \item contain $v$ and at least another point from $A_0$,
    \item are formed by $v$ and at least two points on $C_i$ for some $i\in[r]$,
    \item are formed by $v$ and $k-1$ points from pairwise distinct circles,
\end{itemize}
respectively. The number of regular $(k-1)$-simplices containing $v$ is then at most $N_1+N_2+N_3$.

Since $|A_0|=o(n)$, we have $N_1=o(n^{k-1})$.

For each $i\in[r]$, observe that the number of pairs of points in $X\cap C_i$ that can form a regular $(k-1)$-simplex with $v$ is at most $|X\cap C_i|\leq n$. Indeed, for any point $x\in C_i$ there are at most two other points in $C_i$ at distance $||v-x||$ from $x$. Consequently, each point in $X\cap C_i$ is contained in at most two such pairs, and hence the total number of such pairs is at most $|X\cap C_i|\leq n$. It follows that $N_2=O(n^{k-2})$.

To bound $N_3$, first note that the distance between any two points on different circles is $\sqrt{2\gamma^2}$. Hence, $N_3$ is upper bounded by the number of $(k-1)$-tuples $(x_1,\dots,x_{k-1})\in C_{\ell_1}\times\dots\times C_{\ell_{k-1}}$ with $||v-x_1||=\dots=||v-x_{k-1}||=\sqrt{2\gamma^2}$, over all distinct $\ell_1,\dots,\ell_{k-1}\in[r]$.

If $||v||\neq\gamma$, then the following holds. For any $i\in[r]$, either $i\in\mathcal{I}(v)$, which means that all points on $C_i$ have distance $\sqrt{||v||^2+\gamma^2}\neq\sqrt{2\gamma^2}$ to $v$, or $i\notin\mathcal{I}(v)$, implying that there are at most $2$ points on $C_i$ at distance $\sqrt{2\gamma^2}$ to $v$. So in this case we have $N_3=O(1)$.

If $|\mathcal{I}(v)|\leq r-2$, since for each $j\in[r]\setminus\mathcal{I}(v)$ there are at most $2$ points in $C_j$ with distance $\sqrt{2\gamma^2}$ from $v$, we have
\[N_3\leq\left(\sum_{\mathcal{J}\in\binom{\mathcal{I}(v)}{k-1}}\prod_{j\in\mathcal{J}}|X\cap C_j|\right)+O(n^{k-2})\leq\binom{r-2}{k-1}\left(\frac{n}{r}\right)^{k-1}+o(n^{k-1}).\]

If $||v||=\gamma>0$ and $|\mathcal{I}(v)|\geq r-1$, then $v$ is not the origin and thus $|\mathcal{I}(v)|=r-1$. Let $j\in[r]$ be such that $\mathcal{I}(v)=[r]\setminus\{j\}$. Then $v\in\mathrm{Aff}(C_j)$ and $||v||=\gamma$, implying that $v\in C_j$, a contradiction.

Therefore, $v$ is contained in at most
\[N_1+N_2+N_3\leq\binom{r-2}{k-1}\left(\frac{n}{r}\right)^{k-1}+o(n^{k-1})\]
regular $(k-1)$-simplices. Now we take an arbitrary point $v'\in C_r\setminus X$. Since $v'$ is contained in at least \[\sum_{\mathcal{J}\in\binom{[r-1]}{k-1}}\prod_{j\in\mathcal{J}}|X\cap C_j|=\binom{r-1}{k-1}\left(\frac{n}{r}\right)^{k-1}-o(n^{k-1})>N_1+N_2+N_3\] regular $(k-1)$-simplices, $X\cup\{v'\}\setminus\{v\}$ spans more regular $(k-1)$-simplices than $X$, contradicting the extremality of $X$.
\end{proof}

We finalize the proof by bounding the number of regular $(k-1)$-simplices in $X$. There are three types of simplices:
\begin{enumerate}
    \item regular $(k-1)$-simplices $\Delta_1$ with $|\Delta_1\cap C_i|\leq1$ for all $i\in[r]$;
    \item regular $(k-1)$-simplices $\Delta_2$ with $|\Delta_2\cap C_i|=2$ for some $i\in[r]$ and $|\Delta_2\cap C_i|\leq2$ for others;
    \item regular $(k-1)$-simplices $\Delta_3$ containing at least three vertices on the same circle.
\end{enumerate}
Note that $\Delta_1$ and $\Delta_2$ both have side length $\sqrt{2\gamma^2}$, while $\Delta_3$, if exists, has side length $\sqrt{3\gamma^2}$. 

Recall that for $k\geq3$
\begin{align*}
f_k(n_1,\dots,n_r):=&\left(\sum_{\mathcal{I}\in\binom{[r]}{k}}\prod_{i\in\mathcal{I}}n_i\right)
+\left(\sum_{1\leq\ell\leq\lfloor{k/2}\rfloor}\sum_{\mathcal{J}\in\binom{[r]}{\ell}}\left(\prod_{j\in\mathcal{J}}(n_j-\mathds{1}_{n_j\notin4\ZZ})\right)\left(\sum_{\mathcal{I}\in\binom{[r]\setminus\mathcal{J}}{k-2\ell}}\prod_{i\in\mathcal{I}}n_i\right)\right)\\
&\quad\quad+\mathds{1}_{k=3}\left(\sum_{i\in[r]}\frac{n_i-p_i}{3}+\mathds{1}_{p_i>8}(p_i-8)\right),
\end{align*}
where $p_i$ is the remainder of $n_i$ divided by $12$ for each $i\in[r]$.

\begin{claim}
\label{even_case}
$X$ spans at most $\max_{n_1+\dots+n_r=n}f_k(n_1,\dots,n_r)$ regular $(k-1)$-simplices.
\end{claim}
\begin{proof}[Proof of~\Cref{even_case}]
For every $i\in[r]$, we let $n_i=|X\cap C_i|$ and $p_i$ be the remainder of $n_i$ divided by $12$. Note that $\sum_{i\in[r]}n_i=n$ and $n_i=n/r-o(n)$ for all $i\in[r]$. The number of $\Delta_1$ is
\[\sum_{\mathcal{I}\in\binom{[r]}{k}}\prod_{i\in\mathcal{I}}n_i.\]

We shall upper bound the number of $\Delta_2$ and $\Delta_3$ together as follows. Observe that the third-type simplex $\Delta_3$ exists only when $k=3$. Indeed, when $k\geq4$, since a circle does not contain a regular $(k-1)$-simplex, there must be another point on a different circle, implying that $\sqrt{2\gamma^2}$ and $\sqrt{3\gamma^2}$ both appear as distances between vertices of $\Delta_3$, a contradiction. For each $i\in[r]$ we let $s_i$ be the number of pairs of points on $C_i$ at distance $\sqrt{2\gamma^2}$ apart and let $t_i$ be the number of equilateral triangles spanned by points on $C_i$. Then for $1\leq\ell\leq\lfloor{k/2}\rfloor$, the number of $\Delta_2$ containing exactly $\ell$ pairs of points, each pair on a different circle, is
\[\sum_{\mathcal{J}\in\binom{[r]}{\ell}}\left(\prod_{j\in\mathcal{J}}s_j\right)\left(\sum_{\mathcal{I}\in\binom{[r]\setminus\mathcal{J}}{k-2\ell}}\prod_{i\in\mathcal{I}}n_i\right).\]
The number of $\Delta_3$, which only exists if $k=3$, is $\mathds{1}_{k=3}\sum_{i\in[r]}t_i$. Hence, the number of $\Delta_2$ and $\Delta_3$ together is
\[\left(\sum_{1\leq\ell\leq\lfloor{k/2}\rfloor}\sum_{\mathcal{J}\in\binom{[r]}{\ell}}\left(\prod_{j\in\mathcal{J}}s_j\right)\left(\sum_{\mathcal{I}\in\binom{[r]\setminus\mathcal{J}}{k-2\ell}}\prod_{i\in\mathcal{I}}n_i\right)\right)+\mathds{1}_{k=3}\left(\sum_{i\in[r]}t_i\right).\]

Note that increasing $s_i$ by $1$ increases the total number of $\Delta_2$ and $\Delta_3$ by $\Omega(n^{k-2})$, while increasing $t_i$ by $1$ only increases the total number of $\Delta_2$ and $\Delta_3$ by at most $1$. In order to maximize the total number of $\Delta_2$ and $\Delta_3$, on each circle $C_i$ the priority is to maximize $s_i$, followed by $t_i$. Observe that $s_i\leq n_i$ and the equality can be attained only if $n_i\in4\mathbb{Z}$. Accordingly, the optimal configuration of $n_i$ points on $C_i$ is as described in~\Cref{lb_even}, where $s_i=n_i-\mathds{1}_{n_i\notin4\ZZ}$ is maximized and $t_i=(n_i-p_i)/3+\mathds{1}_{p_i>8}(p_i-8)$. Therefore, the number of $\Delta_2$ and $\Delta_3$ together is at most
\[\left(\sum_{1\leq\ell\leq\lfloor{k/2}\rfloor}\sum_{\mathcal{J}\in\binom{[r]}{\ell}}\left(\prod_{j\in\mathcal{J}}(n_j-\mathds{1}_{n_j\notin4\ZZ})\right)\left(\sum_{\mathcal{I}\in\binom{[r]\setminus\mathcal{J}}{k-2\ell}}\prod_{i\in\mathcal{I}}n_i\right)\right)+\mathds{1}_{k=3}\left(\sum_{i\in[r]}\frac{n_i-p_i}{3}+\mathds{1}_{p_i>8}(p_i-8)\right).\]

Adding up the contributions from different types, we obtain that there are at most
\begin{align*}
&\left(\sum_{\mathcal{I}\in\binom{[r]}{k}}\prod_{i\in\mathcal{I}}n_i\right)
+\left(\sum_{1\leq\ell\leq\lfloor{k/2}\rfloor}\sum_{\mathcal{J}\in\binom{[r]}{\ell}}\left(\prod_{j\in\mathcal{J}}(n_j-\mathds{1}_{n_j\notin4\ZZ})\right)\left(\sum_{\mathcal{I}\in\binom{[r]\setminus\mathcal{J}}{k-2\ell}}\prod_{i\in\mathcal{I}}n_i\right)\right)\\
&\quad\quad+\mathds{1}_{k=3}\left(\sum_{i\in[r]}\frac{n_i-p_i}{3}+\mathds{1}_{p_i>8}(p_i-8)\right)\\
&\leq\max_{n_1+\dots+n_r=n}f_k(n_1,\dots,n_r).
\end{align*}
regular $(k-1)$-simplices in $X$.
\end{proof}

Due to the extremality of $X$,~\Cref{even_case} gives an upper bound on $S^k_{2r}(n)$ when $n$ is sufficiently large. Since the upper bound matches the lower bound~\eqref{lower_bound_even}, we have that
\begin{equation}
\label{max}
S^k_{2r}(n)=\max_{n_1+\dots+n_r=n}f_k(n_1,\dots,n_r)\quad\text{for fixed $r\geq k\geq 3$ and sufficiently large $n$}.
\end{equation}
The proof of~\Cref{simplex_even} is completed by taking $k\geq4$.
\end{proof}

\subsection{The special case: equilateral triangles}
In this subsection we prove~\Cref{triangle_even}.
\begin{proof}[Proof of~\Cref{triangle_even}]
As shown in the proof of~\Cref{simplex_even}, see~\eqref{max}, for sufficiently large $n$ we have $T_{2r}(n)=S^3_{2r}(n)=\max_{n_1+\dots+n_r=n}f_3(n_1,\dots,n_r)$, where
\[f_3(n_1,\dots,n_r)=\left(\sum_{1\leq i<j<k\leq r}n_in_jn_k\right)+\left(\sum_{i\in[r]}(n_i-\mathds{1}_{n_i\notin4\ZZ})(n-n_i)+\frac{n_i-p_i}{3}+\mathds{1}_{p_i>8}(p_i-8)\right)\]
and $p_i$ denotes the remainder of $n_i$ divided by $12$. Thus, to prove~\Cref{triangle_even}, it suffices to find the maximizer of $f_3(n_1,\dots,n_r)$ among all $(n_1,\dots,n_r)$ with $\sum_{i=1}^rn_i=n$.

Let $(n_1,\dots,n_r)$ be a maximizer of $f_3(n_1,\dots,n_r)$ among all $(n_1,\dots,n_r)$ with $\sum_{i=1}^rn_i=n$. Since $S^3_{2r}(n)=\binom{r}{3}(n/r)^3+o(n^3)$ by~\Cref{simplex}, we have that $|n_i-n/r|=o(n)$ for all $i\in[r]$. Otherwise by~\Cref{parts}, the first term of $f_3(n_1,\dots,n_r)$ would be at most $\binom{r}{3}(n/r)^3-\eps n^3$ for some constant $\eps>0$, yielding that $f_3(n_1,\dots,n_r)\leq\binom{r}{3}(n/r)^3-\eps n^3+O(n^2)<S^3_{2r}(n)$, a contradiction.

\begin{claim}
\label{claim1}
For any distinct $i,j\in[r]$, $|n_i-n_j|\leq2$.
\end{claim}
\begin{proof}[Proof of~\Cref{claim1}]
Suppose for contradiction that $n_1-n_2\geq3$. Let $n_1'=n_1-2$ and $n_2'=n_2+2$. Then we have
\begin{align*}
&f_3(n_1',n_2',n_3,\dots,n_r)-f_3(n_1,n_2,n_3,\dots,n_r)\\
&\geq\left(\sum_{k=3}^{r}(n_1'n_2'-n_1n_2)n_k\right)+\left((n_1'-1)(n-n_1')-n_1(n-n_1)\right)\\
&\quad\quad+\left((n_2'-1)(n-n_2')-n_2(n-n_2)\right)-30\\
&=(2n_1-2n_2-4)(n-n_1-n_2)+(-3n+5n_1-6)+(n-3n_2-2)-30\\
&=(2n-2n_1+9)n_1-(2n-2n_2-1)n_2-6n-38.
\end{align*}
Since $n_2\leq n_1-3\leq n/2-1$ and $(2n-2n_2-1)n_2$ is increasing in $n_2$ for $n_2\leq n/2-1$, it holds that $(2n-2n_2-1)n_2\leq(2n-2n_1-7)(n_1-3)$ and thus
\[f_3(n_1',n_2',n_3,\dots,n_r)-f_3(n_1,n_2,n_3,\dots,n_r)\geq10n_1-59>0,\]
contradicting the maximality of $f_3(n_1,\dots,n_r)$.
\end{proof}

\begin{claim}
\label{claim2}
For any distinct $i,j\in[r]$, if $|n_i-n_j|=2$, then $n_i,n_j\in2\ZZ$.
\end{claim}
\begin{proof}[Proof of~\Cref{claim2}]
Suppose that $n_1=n_2+2$ and $n_1,n_2\notin2\ZZ$. Let $n_1'=n_1-1$ and $n_2'=n_2+1$. Since $n_1,n_2\notin4\ZZ$, we have
\begin{align*}
&f_3(n_1',n_2',n_3,\dots,n_r)-f_3(n_1,n_2,n_3,\dots,n_r)\\
&\geq\left(\sum_{k=3}^{r}(n_1'n_2'-n_1n_2)n_k\right)+\left((n_1'-1)(n-n_1')-(n_1-1)(n-n_1)\right)\\
&\quad\quad+\left((n_2'-1)(n-n_2')-(n_2-1)(n-n_2)\right)-30\\
&=(n-n_1-n_2)+(-n+2n_1-2)+(n-2n_2)-30=n-2n_2-30>0,
\end{align*}
a contradiction.
\end{proof}

\begin{claim}
\label{claim3}
There exists at most one $i\in[r]$, such that $n_i$ is odd.
\end{claim}
\begin{proof}[Proof of~\Cref{claim3}]
Suppose that $n_1$ and $n_2$ are both odd. Then by~\Cref{claim1} and~\Cref{claim2} we have that $n_1=n_2$. Let $n_1'=n_1-1$ and $n_2'=n_2+1$. If $n_1=n_2\equiv1\pmod{4}$, then we have $n_1'\equiv0\pmod{4}$ and thus
\begin{align*}
&f_3(n_1',n_2',n_3,\dots,n_r)-f_3(n_1,n_2,n_3,\dots,n_r)\\
&\geq\left(\sum_{k=3}^{r}(n_1'n_2'-n_1n_2)n_k\right)+\left(n_1'(n-n_1')-(n_1-1)(n-n_1)\right)\\
&\quad\quad+\left((n_2'-1)(n-n_2')-(n_2-1)(n-n_2)\right)-30\\
&=-(n-n_1-n_2)+(n_1-1)+(n-2n_2)-30=2n_1-n_2-31>0,
\end{align*}
a contradiction. If $n_1=n_2\equiv3\pmod{4}$, then we have $n_2'\in4\ZZ$ and thus
\begin{align*}
&f_3(n_1',n_2',n_3,\dots,n_r)-f_3(n_1,n_2,n_3,\dots,n_r)\\
&\geq\left(\sum_{k=3}^{r}(n_1'n_2'-n_1n_2)n_k\right)+\left((n_1'-1)(n-n_1')-(n_1-1)(n-n_1)\right)\\
&\quad\quad+\left(n_2'(n-n_2')-(n_2-1)(n-n_2)\right)-30\\
&=-(n-n_1-n_2)+(-n+2n_1-2)+(2n-3n_2-1)-30=3n_1-2n_2-33>0,
\end{align*}
a contradiction.
\end{proof}

Combining~\Cref{claim1},~\Cref{claim2}, and~\Cref{claim3}, we know that
\begin{itemize}
    \item if $n$ is even, then $n_1,\dots,n_r$ are all even and $|n_i-n_j|\leq2$ for all $i,j\in[n]$;
    \item if $n$ is odd, then there is exactly one odd $n_i$ and $|n_i-n_j|=1$ for all $j\in[r]\setminus\{i\}$. 
\end{itemize}
Such a partition of $n$ is unique up to the ordering of $n_1,\dots,n_r$, thus, $(n_1,\dots,n_r)$ is exactly the one given in~\Cref{triangle_even}, up to a reordering.
\end{proof}

\subsection{The lower-order term in $S^k_{2r}(n)$}
In this subsection we prove~\Cref{error_term}, which determines the magnitude of the lower-order term in the asymptotic expression of $S^k_{2r}(n)$.
Recall that for any fixed integers $r\geq k\geq 3$,~\Cref{error_term} asserts that
\[S^k_{2r}(n)=\binom{r}{k}\left(\frac{n}{r}\right)^k+\Theta(n^{k-1}).\]
\begin{proof}[Proof of~\Cref{error_term}]
We have shown in~\eqref{max} that
\[S^k_{2r}(n)=\max_{n_1+\dots+n_r=n}f_k(n_1,\dots,n_r)\quad\text{for fixed $r\geq k\geq 3$ and sufficiently large $n$},\]
where
\[f_k(n_1,\dots,n_r)=\left(\sum_{\mathcal{I}\in\binom{[r]}{k}}\prod_{i\in\mathcal{I}}n_i\right)
+\left(\sum_{j\in[r]}\sum_{\mathcal{I}\in\binom{[r]\setminus\{j\}}{k-2}}(n_j-\mathds{1}_{n_j\notin4\ZZ})\prod_{i\in\mathcal{I}}n_i\right)+O(n^{k-2}).\]
Ignoring the last term above, which has a negligible effect for sufficiently large $n$, it suffices to show that for fixed integers $r\geq k\geq 3$
\begin{equation}
\label{error}
\max_{n_1+\dots+n_r=n}\left(\sum_{\mathcal{I}\in\binom{[r]}{k}}\prod_{i\in\mathcal{I}}n_i\right)
+\left(\sum_{j\in[r]}\sum_{\mathcal{I}\in\binom{[r]\setminus\{j\}}{k-2}}(n_j-\mathds{1}_{n_j\notin4\ZZ})\prod_{i\in\mathcal{I}}n_i\right)=\binom{r}{k}\left(\frac{n}{r}\right)^k+\Theta(n^{k-1}).
\end{equation}
The first term on the left-hand side of~\eqref{error} is at most $\binom{r}{k}(n/r)^k$ by Maclaurin's inequality~\cite{BP03}, while the second term is $O(n^{k-1})$. So the left-hand side of~\eqref{error} is upper bounded by $\binom{r}{k}(n/r)^k+O(n^{k-1})$.

For the other direction, we let $n_i$ be either $\lceil{n/r}\rceil$ or $\lfloor{n/r}\rfloor$ for each $i\in[r]$ so that $n_1+\dots+n_r=n$. Let $m^+$ and $m^-$ denote the counts of indices $i$ for which $n_i=\lceil{n/r}\rceil$ and $n_i=\lfloor{n/r}\rfloor$, respectively. Note that $\alpha:=m^+(\lceil{n/r}\rceil-n/r)+m^-(\lfloor{n/r}\rfloor-n/r)=0$. For this choice of $n_1,\dots,n_r$, the second term on the left-hand side of~\eqref{error} is $\Omega(n^{k-1})$, while the first term is at least
\[\binom{r}{k}\left(\frac{n}{r}\right)^k+\alpha\binom{r-1}{k-1}\left(\frac{n}{r}\right)^{k-1}-O(n^{k-2})=\binom{r}{k}\left(\frac{n}{r}\right)^k-O(n^{k-2}).\]
Therefore, the left-hand side of~\eqref{error} is lower bounded by $\binom{r}{k}(n/r)^k+\Omega(n^{k-1})$.
\end{proof}

\section{Concluding remarks}
\label{sec:concluding}
In this paper, we settled a conjecture in combinatorial geometry. Assuming $n$ being sufficiently large, we obtained the exact value of $T_d(n)$, that is, the maximum number of equilateral triangles spanned by $n$ points in $\RR^d$, for all even dimensions $d\geq6$. For odd $d\geq7$, we determined $T_d(n)$ up to the lower-order term.

Let $T_d^{\rm unit}(n)$ be the maximum number of unit equilateral triangles spanned by $n$ points in $\RR^d$. Our method also determines $T_d^{\rm unit}(n)$ for all even $d\geq6$.
\begin{proposition}
\label{unit_triangle}
Let $r\geq3$ be a fixed integer. For sufficiently large $n\in\NN$ we have
\[T_{2r}^{\rm unit}(n)=\sum_{1\leq i<j<k\leq r}n_in_jn_k+\sum_{i\in[r]}\left((n_i-\mathds{1}_{n_i\notin4\ZZ})(n-n_i)\right),\]
where $(n_1,\dots,n_r)$ is chosen in the same manner as in~\Cref{triangle_even}.
\end{proposition}
The upper bound in~\Cref{unit_triangle} follows by an argument analogous to that of~\Cref{triangle_even,simplex_even}, except that the third-type simplices (triangles) need not be counted, as their side length differs from that of the first two types. The lower bound is obtained by the same Lenz construction as in~\Cref{lb_even}, again excluding the third-type triangles.

We also considered a more general setting. Recall that $S^k_d(n)$ denotes the maximum number of regular $(k-1)$-simplices spanned by $n$ points in $\RR^d$. We determined the asymptotic behavior of $S^k_d(n)$ for all fixed $d\geq2k\geq6$. For even $d$, we are also able to determine the magnitude of the lower-order term and reduce the problem of determining the exact value of $S^k_d(n)$ to an optimization problem. We remark that in the case when $k\geq4$ and $d$ is even, since the third-type simplices do not exist, our proof of~\Cref{simplex_even} also implies the following.
\begin{proposition}
\label{given_side_length}
For any fixed $r\geq k\geq4$ and sufficiently large $n\in\NN$, the maximum number of $(k-1)$-simplices spanned by $n$ points in $\RR^{2r}$, with a given side length, is equal to $S^k_{2r}(n)$.
\end{proposition}

Regarding the function $S^k_{2r+1}(n)$ with $r\geq k\geq3$, it seems out of reach to determine its exact value even for $r=k=3$. The reason is that in the odd-dimensional case a linear proportion of the points can lie on a $2$-dimensional sphere, and the exact maximum number of occurrences of a fixed distance among points on a $2$-dimensional sphere is unknown. Nevertheless, we believe that determining the magnitude of the lower-order term is plausible. As mentioned in~\Cref{lb_odd}, a result of Erd\H{o}s, Hickerson, and Pach~\cite{EHP89} shows that there exists a configuration of $m$ points on a unit $2$-dimensional sphere in which the distance $\sqrt{2}$ occurs $\Omega(m^{4/3})$ times. Building on this configuration, we obtained the lower bound $S^k_{2r+1}(n)\geq \binom{r}{k}\left(\frac{n}{r}\right)^k+\Omega(n^{k-2/3})$, see~\eqref{lower_bound_odd}. On the other hand, a classical result of Clarkson, Edelsbrunner, Guibas, Sharir, and Welzl~\cite{CEGSW90} states that the number of occurrences of any fixed distance among $m$ points on a $2$-dimensional sphere is at most $O(m^{4/3})$. Using this bound, together with a further refinement of our method, it is possible to prove that
\[S^k_{2r+1}(n)=\binom{r}{k}\left(\frac{n}{r}\right)^k+\Theta(n^{k-2/3})\]
for $r\geq k\geq3$. However, we expect the analysis to be more involved, since most of the dimension arguments underlying the stability result (i.e.,~\Cref{simplex_stability}) no longer hold in the odd-dimensional case.

\section*{Acknowledgements}
The first author is partially supported by a PIMS Postdoctoral Fellowship (PIMS-20250726-PDF). The authors thank the anonymous referee for pointing out  errors in earlier versions of the manuscript and for many constructive comments that significantly improved the presentation.

\end{document}